\documentclass[a4paper,12pt]{article}

\def\Max{{\rm Max}}
\usepackage[english]{babel}
\usepackage{amsthm,amsmath,amssymb}
\usepackage{mathtools}
\usepackage{hyperref}

\DeclarePairedDelimiter\floor{\lfloor}{\rfloor}
\newtheorem{rema}{Remark}

\newtheorem{claim}{Claim}
\newtheorem{lemma}{Lemma}
\newtheorem{corollary}{Corollary}
\newtheorem{prop}{Proposition}
\newtheorem{thm}{Theorem}

\def\M{{\cal M}}
\def\ds{\displaystyle}

\def \g   {\gamma}
\def \d   {\delta}

\def \eps {\varepsilon}
\def\E{{\mathbb{E}}}
\def\P{{\mathbb{P}}}
\def\R{{\mathbb{R}}}
\def\N{{\mathbb{N}}}

\def\tF{\tilde{\cal{F}}}
\newcommand{\supp}{\mathop{\mathrm{supp}}}

\def\|{\,|\, }

\title{A local barycentric version of the Bak-Sneppen model}
\author{Philip Kennerberg${}^*$ and Stanislav Volkov\footnote{Centre for Mathematical Sciences, Lund University, Box 118 SE-22100, Lund, Sweden}}
\begin{document}

\maketitle
\begin{abstract}
We study the behaviour of an interacting particle system, related to the Bak-Sneppen model and Jante's law process defined in~\cite{KV1}. Let $N\ge 3$ vertices be placed on a circle, such that each vertex has exactly two neighbours. To each vertex assign a real number, called {\em fitness}\footnote{we use this term, as it is quite standard for Bak-Sneppen models}. Now find the vertex which fitness deviates most from the average of the fitnesses of its two immediate neighbours (in case of a tie, draw uniformly among such vertices), and replace it by a random value drawn independently according to some distribution $\zeta$. We show that in case where $\zeta$ is a finitely supported or continuous uniform distribution, all the fitnesses except one converge to the same value.
\end{abstract}

\noindent{\sf Keywords}: Bak-Sneppen model, Jante's  law process, interacting particle systems.

\noindent{\sf Subject classification:} 60J05, 60K35,  91D10.

\section{Introduction}
The model we study in the current paper is a ``marriage" between  Jante's law process and the Bak-Sneppen model.

Jante's law process refers to the interacting particle model studied in~\cite{GVW} under the name ``Keynesian beauty contest process'', and generalized in~\cite{KV1}. This model runs as follows. Fix an integer $N\ge 3$, $d\ge 1$, and some $d$-dimensional random variable $\zeta$. Let the initial configuration consist of $N$ arbitrary points in $\R^d$. The process runs in discrete time according to the following algorithm: first, compute the centre of mass $\mu$  of the given configuration of~$N$ points; then replace the point which is the most distant from~$\mu$ by a new $\zeta-$distributed point drawn independently each time. It was shown in~\cite{GVW} that if~$\zeta$ has a uniform distribution on the unit cube, then all but one points converge to some random point in~$\R^d$. This result was further generalized in~\cite{KV1}, by allowing $\zeta$ to have an arbitrary distribution, and additionally removing not just $1$, but $K\ge 1$ points chosen to minimize a certain functional. The term ``Jante's law process" was also coined in~\cite{KV1}, to reflect that this process is reminiscent of the ``Law of Jante'' principle,  which describes patterns of group behaviour towards individuals within Scandinavian countries that criticises individual success and achievement as unworthy and inappropriate; in other words, it is better to be ``like everyone else''. The origin of this ``law" dates back to Aksel Sandemose ~\cite{AS}. Another modification of this model in one dimension, called the $p$-contest, was introduced in~\cite{GVW,HCW} and later studied e.g.\ in~\cite{KV2}. This model runs as follows: fix some constant $p\in(0,1)\cup(1,\infty)$, and replace the point which is the farthest from $p\mu$ (rather than $\mu$). 

Finally, we want to mention that the phenomenon of conformity is observed in many large social networks, see, for example, \cite{BDB,MN,TWS} and references therein.

Pieter Trapman (2018, personal communications) suggested to study  Jante's law model with {\em local interactions}, thus making it somewhat similar to the famous Bak-Sneppen (BS) model see e.g.~\cite{BS}. In the BS model,~$N$ species are located around a circle, and each of them is associated with a so-called ``fitness",  which is a real number. The algorithm consists in choosing the least fit individual, and then replacing it {\em and both of its two closest neighbours} by a new species, with a new random and independent fitness. After a long time, there will be a minimum fitness, below which species do not survive. The model proceeds through certain events, called ``avalanches", until it reaches a state of relative stability where all fitnesses are above a certain threshold level. There is a version of the model where fitnesses take only values $0$ and $1$ (see~\cite{BK} and~\cite{SVBS}), but even this simplified version turns out to be notoriously difficult to analyse, see e.g.~\cite{MZ}. Some more recent results can be found in~\cite{Ben,Veer}.

The barycentric Bak-Sneppen model, or, equivalently, Jante's law process with local interactions, is defined as follows. Unlike the classical Bak-Sneppen model, our model is based on some {\em local phenomena}, which makes it much more tractable mathematically, and hence we are able to obtain substantial rigorous results.

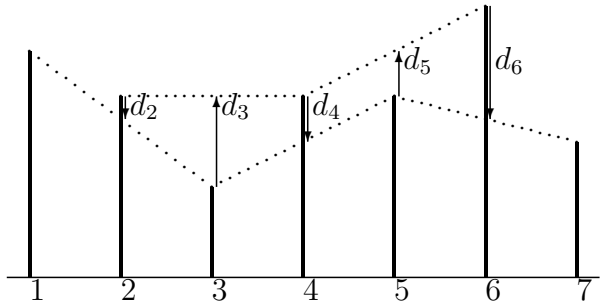
\begin{figure}
  \centering
\setlength{\unitlength}{6mm}
\begin{picture}(11,11)
\linethickness{0.2mm}
\put(-0.5,0){\line(1,0){13}}
\linethickness{0.5mm}
 \put(0,0){\line(0,1){5}}
 \put(2,0){\line(0,1){4}}
 \put(4,0){\line(0,1){2}}
 \put(6,0){\line(0,1){4}}
 \put(8,0){\line(0,1){4}}
\put(10,0){\line(0,1){6}}
\put(12,0){\line(0,1){3}}
\multiput(0,5)(0.2,-0.15){21}{\circle*{0.05}}
\multiput(2,4)(0.2,0)    {21}{\circle*{0.05}}
\multiput(4,2)(0.2,0.1)  {21}{\circle*{0.05}}
\multiput(6,4)(0.2,0.1)  {21}{\circle*{0.05}}
\multiput(8,4)(0.2,-0.05){21}{\circle*{0.05}}
\linethickness{0.2mm}
\put(2.1,4){\vector(0,-1){0.5}}
\put(4.1,2){\vector(0,1){2}}
\put(6.1,4){\vector(0,-1){1}}
\put(8.1,4){\vector(0,1){1}}
\put(10.1,6){\vector(0,-1){2.5}}
\put(2.2,3.6){$d_2$}
\put(4.2,3.6){$d_3$}
\put(6.2,3.6){$d_4$}
\put(8.2,4.6){$d_5$}
\put(10.2,4.6){$d_6$}
\put(0,-0.5){1}
\put(2,-0.5){2}
\put(4,-0.5){3}
\put(6,-0.5){4}
\put(8,-0.5){5}
\put(10,-0.5){6}
\put(12,-0.5){7}
\end{picture}
\caption{Illustration of the distances from the average of the two neighbours; $\jmath=6$.}
\label{fig1}
\end{figure}

Fix an integer $N\ge 3$, and let $S=\{1,2,\dots,N\}$ be the set of nodes uniformly spaced on a circle.  At time $t$, each node $i\in S$ has a certain ``fitness" $X_i(t)\in\R$; let $X(t)=(X_1(t),\dots,X_N(t))$.  Next, for the vector  $x=(x_1,\dots,x_N)$, define
$$
d_i(x)=\left|x_i-\frac{x_{i+1}+x_{i-1}}2\right|, 
$$
as the measure of local ``non-conformity'' of the fitness at node $i$ (here and further we will use the convention that   $N+1\equiv 1$, $N+2\equiv 2$, and $1-1\equiv N$ for indices on $x$). Let also $d(x)=\max_{i\in S} d_i(x)$.

The process runs as follows. Let $\zeta$ be some fixed one-dimensional random variable. At time $t$, $t=0,1,2,\dots$, we chose the ``least conformist node''\footnote{The intuition for choosing the deviance as the criteria for removal is the follows. In many Scandinavian countries, non-conformity is considered as a very bad treat, and as a result, individuals which divert from the average, tend to be less successful in these societies. This phenomenon is called ``The Jante's Law''. We understand that the word ``fitness'' is thus somewhat misleading here, but would like to use it to keep in line with the standard Bak-Sneppen model.} $i$, i.e.\ the one maximizing $d_i(X(t))$, and replace it by a $\zeta$-distributed random variable. By $\jmath(x)$ we denote the index of such a node in the configuration $x=(x_1,\dots,x_N)$, that is
$$
d_{\jmath(x)}(x)=d(x)
$$
(see Figure~\ref{fig1}). If there is more than one such node, we choose any of them with equal probability, thus $\jmath(x)$ is, in general, a random variable. Also assume that all the coordinates of the initial configuration $X(0)$ lie in the support of $\zeta$. We are interested in the long-term dynamics of this process.

We start with a somewhat easier version of the problem, where $\zeta$ takes finitely many distinct values (Section~\ref{Sec1}), and then extend this result to the case where $\zeta\sim U[0,1]$ (Section~\ref{Sec2}). We will show that all the fitnesses (except the one which has just been updated) converge to the same (random) value. This will hold for each of the two models.

\begin{rema}\label{rem1}
One can naturally extend this model to any finite connected non-oriented graph $G$ with vertex set $V$, as follows. For any two vertices $v,u\in V$ that are connected by an edge we write $u\sim v$.  To each vertex $v$ assign a fitness $x_v\in\R$, and define the measure of non-conformity of this vertex  as
$$
d_v(x)=\left|x_v-\frac{\sum_{u:\ u\sim v} x_u}{N_v}\right|, 
$$
where $N_v=|{u\in V:\, u\sim v}|$ denotes the number of neighbours of $v$, and the replacement algorithm runs exactly as it is described earlier. 

In particular, if $G$ is a cycle graph, we obtain the model studied in the current paper. On the other hand, if $G$ is a complete graph, we obtain the model equivalent to that studied in~\cite{GVW, KV1}.
\end{rema}

\begin{rema}\label{rem2}
Unfortunately,  our results cannot be extended to a general model, described in Remark~\ref{rem1}.
Indeed, assume that $\supp\zeta=\{0,1\}$. It is not hard to show that if for some $v$ we have $N_v=1$, then the statement of Theorem~\ref{thm-dc} does not have to hold. 

Moreover, it turns out that even when all the vertices have at least two neighbours (i.e., $N_v\ge 2$ for all $v\in V$), then there are still counterexamples: please see Figure~\ref{figgengraph}.
\end{rema}

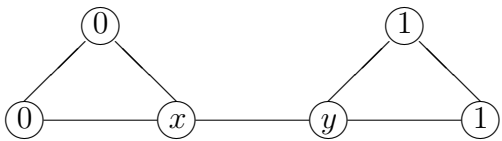
\begin{figure}  \centering
\setlength{\unitlength}{5mm}
\begin{picture}(10,5)
\put(0,0){\circle{1}}
\put(4,0){\circle{1}}
\put(8,0){\circle{1}}
\put(12,0){\circle{1}}
\put(2,2.5){\circle{1}}
\put(10,2.5){\circle{1}}
\put(0.5,0){\line(1,0){3}}
\put(4.5,0){\line(1,0){3}}
\put(8.5,0){\line(1,0){3}}
\put(0,0.5){\line(1,1){1.6}}
\put(8,0.5){\line(1,1){1.6}}
\put(4,0.5){\line(-1,1){1.6}}
\put(12,0.5){\line(-1,1){1.6}}     
\put(-0.2,-0.2){$0$}
\put(3.8,-0.2){$x$}
\put(7.8,-0.2){$y$}
\put(11.8,-0.2){$1$}
\put(1.8,2.3){$0$}
\put(9.8,2.3){$1$}
\end{picture}
\caption{On this graph with $N=6$ vertices, only values $x$ and $y\in\{0,1\}$ are updated all the time; infinitely often half of the fitnesses equal $0$, while the other half equals $1$.}
\label{figgengraph}
\end{figure}

The rest of the paper is organized as follows.
In Section~\ref{Sec1} we study the easier, discrete, case. We show the convergence by explicitly finding all the absorbing classes for the finite-state Markov chain.

Section~\ref{Sec2} contains the main result of our paper, Theorem~\ref{thm_main}, which shows that all but one fitness converge to the same (random) limit, similarly to the main result of~\cite{GVW}.

\newpage

\section{Discrete case}\label{Sec1}
In this Section we study the case when fitnesses take finitely many values, equally spaced between each other. Due to the shift- and scale-invariance of the model, without loss of generality we may assume that $\supp \zeta=\{1,2,\dots,M\}=:\M$, and that $\ds p=\min_{j\in \M} \P(\zeta=j)>0$. In this case~$X(t)$ becomes a finite state-space Markov chain on $\M^N$.

Note that if $N-1$ fitnesses coincide and are equal to some $a\in\M$, then it is the fitness that differs from $a$ that will keep being replaced, until it finally coincides with the others. When this happens, we will have to choose randomly one among all the vertices, and replace its fitness. The replaced fitness may or may not differ from $a$, and then this procedure will repeat over and over again.
Hence, to simplify the rest of the argument, we can (and will) safely modify the process as follows:
$$
X(t+1)\equiv X(t)\text{ as soon as }d(X(t))=0\text{ i.e.\ all $X_i(t)=a$ for some $a\in\M$.}
$$
We will say that the process that the process is {\em absorbed} at value $a$.

\begin{rema}
The fact that the values of $\zeta$ are equally spaced is, surprisingly, crucial.
Let $\supp\zeta=\{0,1,5,6\}=:\M$ and $N=8$. Then the set of configurations
$$
[0,1,x,5,6,5,y,1],\quad x,y\in\M
$$
is stable; the maximum distance from the average of the fitnesses of the neighbours is always at nodes $3$ or $7$, and it equals $2$ or $3$, while the other distances are at most  $1.5$ or $2$ respectively.
\end{rema}

\begin{thm}\label{thm-dc}
The process $X(t)$ gets absorbed at some value $a\in\M$,
regardless of its starting configuration $X(0)\in \M^N$.
\end{thm}
First, observe that since $X(t)$, $t=0,1,2,\dots$ is a finite-state Markov chain on $\M^N$ with the set of absorbing states 
$$
{\rm O}=(1,1,\dots,1) \cup (2,2,\dots,2) \cup\dots (M,M,\dots,M) \subset \M^N
$$
it suffices to show that $\rm O$ is accessible (can be reached with a positive probability in some number of steps) from any starting configuration $X(0)$.

First, for $x=(x_1,x_2,\dots,x_N)\in\M^N$,  define
\begin{align*}
\Max(x)&=\max_{1\le i\le N} x_i,\\
S(x)&=\left\{j\in\{1,2,\dots,N\}:\ x_j=\Max(x)\right\}.
\end{align*}
that is, the maximum of $x$, and the indices of $x$ where this maximum is achieved\footnote{for example, if $x=(1,4,2,4,4,2)$ then $\Max(x)=4$, $S(x)=\{2.4.5\}$.}.
Let us also define 
$$
f(x)=\sum_{i=1}^N (x_i-x_{i+1})^2
$$
with the convention $x_{N+1}\equiv x_1$, which we will use as some sort of Lyapunov function.
The following two algebraic statements are not difficult to prove.
\begin{claim}\label{Claimf=d}
$f(x)=0$ if and only if $d(x)=0$. 
\end{claim}
\begin{proof}
Let $x=(x_1,\dots,x_N)$. One direction is trivial:
if $f(x)=0$, then $x_i\equiv x_1$ for all $i\in S$ and hence $d_i(x)=0$ for all $i\in S$ $\Longleftrightarrow d(x)=0$. 

On the other hand, suppose that $d_i(x)=0$ for all $i$. If not all $x_i$'s are equal, there must be an index $j$ for which $x_j=\max_{i\in S}x_i$, and either $x_{j-1}< x_j$ or $x_{j+1}< x_j$. This, in turn, implies that
$
2d_j(x)=|(x_j -x_{j-1})+(x_j-x_{j+1})|=(x_j -x_{j-1})+(x_j-x_{j+1})>0
$
yielding a contradiction.
\end{proof}

\begin{claim}\label{ClaimfAt}
Let $x=(x_1,\dots,x_{i-1},x_i,x_{i+1},\dots,x_N)$ and 
and $x'=(x_1,\dots,x_{i-1},a,x_{i+1},\dots,x_N)$ where $a=\floor*{\frac{x_{i-1}+x_{i+1}}{2}}$. Then 
\begin{itemize}
\item[(a)] $f(x')\le f(x)$;
\item[(b)] if additionally $d_i(x)\ge 1$ then $f(x')\le f(x)-1$. 
\end{itemize}
\end{claim}

\begin{rema}
One may expect that there are simpler Lyapunov functions; while we cannot rule this out, let us illustrate two natural candidates that, unfortunately, fail. First, consider $d(x)$; however this function does not work as the next example shows.
Let $x= [1, 3, 9, 18, 24, 27, 27, 24, 18, 9, 3, 1]$. Then $d_i(x)$ is the largest at $i=2$ and $i=11$; thus $d(x)=d_2(x)=2$. If we replace a ``3'' by ``4''$=(1+9)/2$, then
$x'= [1, 4, 9, 18, 24, 27, 27, 24, 18, 9, 3, 1]$ so $d(x')=d_3(x)=2.5>d(x)$.

Another possible candidate, \ $\tilde f(x)=\sum_i d_i(x)^2$ does not work either: let $x=[1,6,9,6,1]$, then $x'=[1,6,6,6,1]$ and $\tilde f(x')>  \tilde f(x)$, so it is not a Lyapunov function either.
\end{rema}

\begin{proof}[Proof of Claim~\ref{ClaimfAt}].
From simple algebra it follows that
\begin{align*}
\frac{f(x')-f(x)}2&=(a-x_i)(a+x_i-x_{i-1}-x_{i+1})
\\ &
=\left(a-\frac{x_{i-1}+x_{i+1}}2\right)^2
-\left(x_i-\frac{x_{i-1}+x_{i+1}}2\right)^2
=d_i(x')^2-d_i(x)^2=:(*).
\end{align*}
Note that if $d_i(x)=0$ or $d_i(x)=1/2$ , then $d_i(x')=d_i(x)$ and thus $(*)=0$.
On the other hand, if $d_i(x)\ge 1$, since $d_i(x')\le 1/2$, we get $(*)\le - 1/2$.
\end{proof}

To simplify notations, denote
$$
\jmath_t=\jmath(X(t)),\qquad {\Delta}_t=d(X(t)),\qquad f_t=f(X(t)).
$$
Now we are going to construct an explicit path through which $X(t)$ can reach $\rm O$ starting from any initial state.
Let
\begin{align*}
A_t=\left\{ \right.&X_{\jmath_t}(t)\text{ is replaced by }X_{\jmath_t}(t+1)=\floor*{\frac{X_{\jmath_t-1}(t)+X_{\jmath_t+1}(t)}{2}},\\
& \left.
\text{ and }\jmath_t\in S(X(t))\text{ if possible}\right\}.
\end{align*}
Note that the second condition is always possible to satisfy when $\Delta_t=1/2$. Indeed,  if $\Delta_t=1/2$ for $X(t)=x$, then there must be a $j$ such that $x_j=\Max(x)$ but $x_{j+1}\le \Max(x)-1$.
As a result, $d_j(x)\ge 1/2$ and hence~$x_j$ is one of the points which can be potentially replaced.

Now the statement of Theorem~\ref{thm-dc} will follow from the following Lemma.
\begin{lemma}\label{lem-dc}
For any $X(0)$ there is a $T\ge 0$ such that on the event
$$
A_0\cap A_1\cap ... \cap A_T
$$
we have $X(T)\in {\rm O}$.
\end{lemma}
This Lemma, in turn,  immediately follows from the next statement and the observation that $0\le f(x)\le M^2N$, as well as the fact that $f(X_T)=0\Longleftrightarrow \Delta_T=0 \Longleftrightarrow X_T\in {\rm O}$ (see Claim~\ref{Claimf=d}).

\begin{claim}\label{ClaimB_t}
If $f_s>0$ then $f_{s+N-2}\le f_s-1$ on $A_s \cap A_{s+1} \cap\dots \cap A_{s+N-2}$.
\end{claim}
\begin{proof}
Note that $\Delta_t$ can take only values $\{0,\frac12,1,\frac32,2,\dots\}$.
W.l.o.g.\ we assume that $s=0$. 

First, if $\Delta_t=0$ for some $0\le t\le N-2$, then $f_t=0$ by Claim~\ref{Claimf=d} and  by Claim~\ref{ClaimfAt}(a) and the fact that $f_0\ge 1$, we have $f_{N-2}\le 0=f_t\le f_0-1$. From now on suppose that $\min_{0\le t\le N-2}\Delta_t\ge 1/2$.

We will show that it is impossible to have $\Delta_t= \frac 12$ simultaneously for all $t=0,1,2,\dots,N-3$ (observe that the case  $\Delta_t=1/2$ contains, quite counter-intuitively, a very rich set of states, see Figure~\ref{refer01}). 
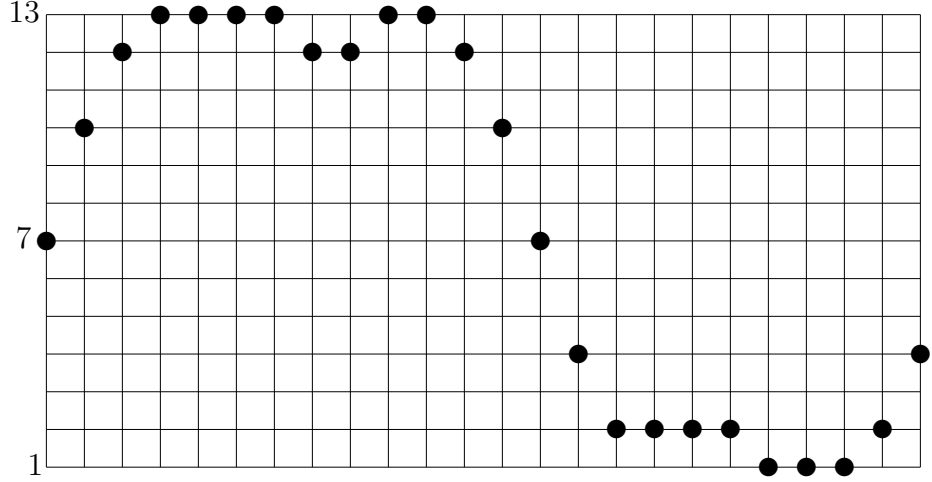
\begin{figure}[h]  \centering
\setlength{\unitlength}{5mm}
\begin{picture}(19,12)
\linethickness{0.01mm}
\multiput(0,0)(1,0){24}{\line(0,1){12}}
\multiput(0,0)(0,1){13}{\line(1,0){23}}
\put(0,6){\circle*{0.5}}
\put(1,9){\circle*{0.5}}
\put(2,11){\circle*{0.5}}
\put(3,12){\circle*{0.5}}
\put(4,12){\circle*{0.5}}
\put(5,12){\circle*{0.5}}
\put(6,12){\circle*{0.5}}
\put(7,11){\circle*{0.5}}
\put(8,11){\circle*{0.5}}
\put(9,12){\circle*{0.5}}
\put(10,12){\circle*{0.5}}
\put(11,11){\circle*{0.5}}
\put(12,9){\circle*{0.5}}
\put(13,6){\circle*{0.5}}
\put(14,3){\circle*{0.5}}
\put(15,1){\circle*{0.5}}
\put(16,1){\circle*{0.5}}
\put(17,1){\circle*{0.5}}
\put(18,1){\circle*{0.5}}
\put(19,0){\circle*{0.5}}
\put(20,0){\circle*{0.5}}
\put(21,0){\circle*{0.5}}
\put(22,1){\circle*{0.5}}
\put(23,3){\circle*{0.5}}
\put(-0.5,-0.2){$1$}
\put(-0.8,5.8){$7$}
\put(-1,11.8){$13$}
\end{picture}
\caption{\em \small A configuration with $\Delta_t = 1/2$ (note the periodic boundary conditions), $\M=\{1,2, \dots, 13\}$ and $N = 24$. Observe that if $\Delta_t=1/2$ then there will be  a number of ``plateaus'' each containing at least two maximal fitnesses; moreover, any two such plateaus will be separated by at least two non-maximal fitnesses.}\label{refer01}
\end{figure}
Indeed, the set $S(X(t))$ of indices of the maximum fitnesses must contain between $2$ and $N-2$ elements\footnote{a single maximum would imply $\Delta_t\ge 1$, the same holds if there are $N-1$ coinciding maxima; finally, $|S(X(t))|=N$ would imply that $\Delta_t=0$.}. However, on~$A_t$ we have $S(X(t+1))\subset S(X(t))$ and $|S(X(t+1))|=|S(X(t))|-1$ by construction. Since $S(X(0))\le N-2$, the value $\Delta_t$ cannot stay equal to $1/2$ for $N-2$ consecutive steps, and thus this case is impossible.

As a result, we conclude that $\Delta_t\ge 1$ for some $t\in\{0,1,\dots, N-3\}$. Then $f_{t+1}\le f_t-1$  by Claim~\ref{ClaimfAt}(b). As a result, $f_{N-2}\le f_{t+1}\le f_t-1\le f_0$ by Claim~\ref{ClaimfAt}(a).
\end{proof}

\begin{rema}
We have actually shown that $T$ in Lemma~\ref{lem-dc} can be chosen no larger than $M^2N\times (N-2)$, i.e.\ $\P(X(M^2N(N-2))\in{\rm O}\| X(0)=x)>0$ for any $x\in\M^N$.
\end{rema}

\begin{rema}
It would be interesting  to find the distribution of the limiting absorbing configuration, i.e.\ $\xi:=\lim_{t\to\infty} X_i(t)$; clearly it will depend on $X(0)$. This  is quite hard problem, and we can present only results based on simulations. Figure~\ref{figsim} shows the histograms of the distribution of $\xi$ for different values of $M$ and $N$, starting from a random initial condition, i.e.\ $X_i(0)$ are i.i.d.\ random variable uniformly distributed on $\M$.
\end{rema}

\begin{figure}[h]
\centering
\includegraphics[scale=0.3]{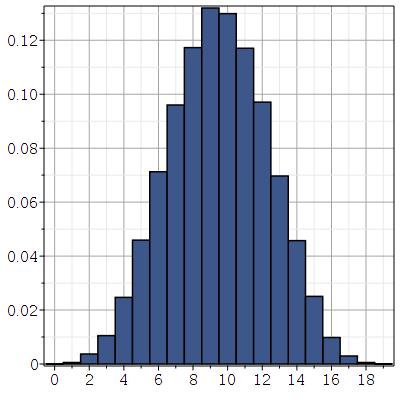}
\includegraphics[scale=0.3]{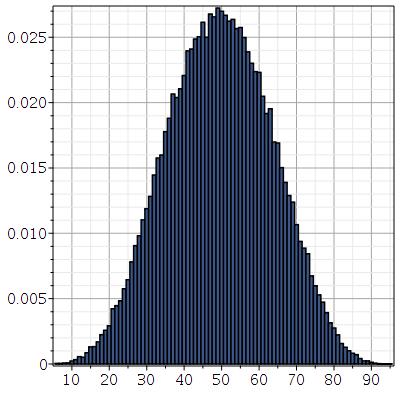}
\includegraphics[scale=0.3]{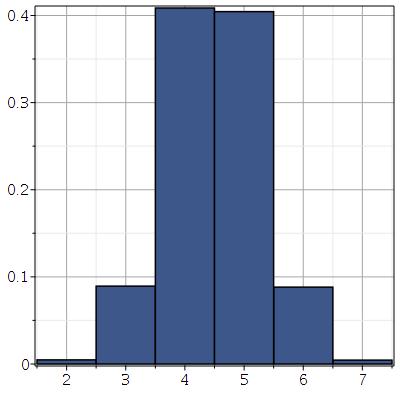}
\caption{Distribution of $\xi$ based on simulations, for $(N,M)=(20,20)$, $(20,100)$, and $(200,10)$ respectively. Uniform random initial conditions.}\label{figsim}
\end{figure}

\section{Continuous case}\label{Sec2}
Throughout this section, we assume that $\zeta\sim U[0,1]$, and $X_i(t)\in [0,1]$ for all $i\in S$ and $t=0,1,2,\dots$. We also assume that $X(0)$ is such that $\jmath(X(0))$ is non-random.

\begin{thm}\label{thm_main}
There exists a.s.\ a random variable  $\bar X\in [0,1]$ such that as $t\to\infty$
$$
(X_1(t),X_2(t),\dots,X_{\jmath(X(t))-1}(t),X_{\jmath(X(t))+1}(t),
\dots,X_N(t))\to (\bar X,\bar X,\dots,\bar X)\in [0,1]^{N-1}\qquad\text{a.s.}
$$
\end{thm}

The proof of this theorem will consists of two parts.
Firstly (see Lemma~\ref{lem_xi->0}), we will show that the properly defined ``spread'' between the values  $X_1(t),\dots,X_N(t)$ converges to zero.
This does not, however, imply the the desired result, as hypothetically we can have the situation best described by the ``Dance of the Little Swans'' from Tchaikovsky's ``Swan Lake'': while the mutual distances between the $X_i$'s decrease or even some stay $0$, their common location changes with time, and thus does not converge to a single point in $[0,1]$. This can happen, for example, if the diameter of the configuration converges to zero too slowly.

The second part of the proof will show that not only the distances between the $X_i$'s decrease, but they all (but the most recently changed one) converge to the same random limit. Please note that the similar strategy was used in~\cite{GVW}, however, in our case both steps require much more work.


It turns out that it is much easier to work with the embedded process, for which either the non-conformity of the node at which the value is replaced, is smaller than the initial non-conformity, or at least the location of the ``worst" node (i.e.\ the one where $d_i$ is the largest) has changed, whichever comes first. Formally, let
$\nu_0=0$ and recursively define for $k=0,1,2,\dots$
$$
\nu_{k+1}=\inf\left\{t>\nu_{k}:\ \jmath(X(t))\ne \jmath(X(\nu_{k}))\text{ or }d(X(t))<d(X(\nu_k))\right\}.
$$
Note that due to the continuity of $\zeta$ each $\jmath(X(t))$ is uniquely defined a.s., and that all $\nu_k$ are finite a.s..

\noindent
{\bf Examples:}
\begin{itemize}
\item[(a)]
$x=(\dots 0.5,\underline{0.6},0.5,0.3,\dots)$. The ``worst" node is the second one (with the fitness of $0.6$) and $d=d_2(x)=0.1$; it is replaced, say, by $0.32$. Now the configuration becomes  
$$
x'=(\dots, 0.5,0.32,\underline{0.5},0.3,\dots)$$ 
and the worst node is the third one with $d(x')=d_3(x')=0.19>0.1=d(x)$; 

\item[(b)] $x$ is the same as in (a), but $x_2$ is replaced by $0.58$. Now the configuration becomes  $$x=(\dots,0.5,\underline{0.58},0.5,0.3,\dots)$$ and the worst node is still the second one with $d(x')=d_2(x')=0.08<0.1=d(x)$.
\end{itemize}

Now let $\tilde X(s)=X(\nu_s)$  and $\tF_s=\sigma\left(\tilde X(1),\dots,\tilde X(s)\right)$ be the filtrations associated with this embedded process. Since throughout time $[\nu_k,\nu_{k+1})$ the value $\jmath$ remains constant at $\jmath_{\nu_k}$ and only~$X_{\jmath_{\nu_k}}$ is updated, we have
$$
X_i(t)=X_i(\nu_k)\quad\text{for all } i\ne \jmath(X(t))
$$
for $t\in [\nu_k,\nu_{k+1})$. Moreover, the process $\tilde X$ evolves as a Markov process but with the ``update'' distribution restricted from the full range, since a uniform distribution conditioned to be in some subinterval is still uniform (this will be used later in Lemma~\ref{lemsuperm}).
Hence Theorem~\ref{thm_main}  follows immediately from
\begin{thm}\label{thm_main_alt}
There exists a.s.\ a random variable $\bar X\in [0,1]$ such that as $s\to\infty$
$$
(\tilde X_1(s),\tilde X_2(s),\dots,\tilde X_N(s))\to (\bar X,\bar X,\dots,\bar X)\in [0,1]^{N}\qquad\text{a.s.}
$$
(Moreover, this convergence happens exponentially fast: there is an $s_0=s_0(\omega)<\infty$ and a non-random $\gamma\in(0,1)$   such that 
$
\left| \tilde X_i(s)-\bar X\right|\le \g^s
$
for all  $i\in S$ and $s\ge s_0$.)
\end{thm}
\begin{rema}
In what follows, we assume that $N\ge 5$. The cases $N=3$ and $N=4$ can be studied somewhat easier, and we leave this as an exercise.
\end{rema}

We will use the Lyapunov functions method, with a clever choice of the function.
For $x=(x_1,x_2,\dots,x_N)$ define
\begin{align*}
h(x)&=2\cdot \sum_{i\in S} (x_i-x_{i+1})^2
+ \sum_{i\in S} (x_i-x_{i+2})^2
=2 \sum_{i\in S} \left(3x_i^2-2 x_i x_{i+1} - x_i x_{i+2}\right).
\end{align*}

We start by showing that $h(\tilde X(s))$ is a non-negative supermartingale (Lemma~\ref{lemsuperm}), hence it must converge a.s. Then we show that this limit is actually $0$ (Lemma~\ref{lem_xi->0}). Combined with the fact that $h(\tilde X(s))$, as a metric, is equivalent to $\max_{i,j} |\tilde X_i(t)-\tilde X_j(t)|$, (see Lemma~\ref{lembounds}) this ensures that eventually all $\tilde X_i$ become very close to each other, thus establishing the first necessary ingredient of the proof of the main theorem.

\begin{lemma}\label{lemsuperm}
$\xi(s)=h\left(\tilde X(s)\right)$ is a non-negative supermartingale.
\end{lemma}
\begin{proof}
The non-negativity of $\xi(s)$ is obvious. To show that it is a supermartingale, assume that $\tilde X(s)=(x_1,x_2,x_3,x_4,x_5,\dots)$ and w.l.o.g.\ that $\jmath(\tilde X(s))=3$. Suppose that the allowed range (i.e., for which either $d$ decreases or the location of the minimum changes) for the newly sampled point is $[a,b]\subseteq [0,1]$. Assuming the newly sampled point is uniformly distributed on $[a,b]$ (since a restriction of the uniform distribution to a subinterval is also uniform), we get
\begin{align}\label{eqDrift}
\Delta&:=\E (\xi(s+1)-\xi(s) |\tF_s)=
\int_a^b \left\{
2(x_2-u)^2+2(u-x_4)^2+(x_1-u)^2+(u-x_5)^2\right.
 \nonumber \\
&\qquad \left. -\left[2(x_2-x_3)^2+2(x_3-x_4)^2+(x_1-x_3)^2+(x_3-x_5)^2\right]\right\}\frac{\rm{d}u}{b-a}
\\ \nonumber
&=2(a^2+b^2+ab)+(2x_3-a-b)(x_1+2x_2+2x_4+x_5)
-6x_3^2.
\end{align}
Now we need to compute the appropriate $a$ and $b$, and then show that $\Delta\le 0$.

W.l.o.g.\ we can assume that $x_3>\frac{x_2+x_4}{2}$, the  case $x_3<\frac{x_2+x_4}{2}$ is equivalent to $(1-x_3)>\frac{(1-x_2)+(1-x_4)}{2}$. Now setting $\tilde x_i=1-x_i$ for all $i$ yields identical calculations.

Suppose that the fitness at node $3$ is replaced by some value $X(\nu_s+1)=:u$, let the new value of the non-conformity at node $3$ be  $d_3'=d_3(x_1,x_2,u,x_4,x_5,\dots)=d_3(X(\nu_s+1))$.

\begin{itemize}
\item
If $x_3$ is replaced by $u>x_3$, then this value will be ``rejected", in the sense that $d$ has only increased while the $\arg\max_{i\in S} d_i$ is still at the same node (i.e.,  $3$). Indeed, when $x_3$ increases by some~$\delta>0$, so does~$d_3$, while $d_2$ and $d_4$ can potentially increase only by $\delta/2$ and thus cannot overtake $d_3$.

\item
When $u\in \left(\frac{x_2+x_4}{2},x_3\right)$,  $d_3'$ is definitely smaller than the original $d_3$.

Assume from now on that $u\in \left(0,\frac{x_2+x_4}{2}\right)$. When $x_3$ is replaced by $u$, it might happen that while the new $d_3$ is larger than the original one, the value of $d_2$ or $d_4$ overtakes $d_3$.

\item
When $u\in \left(0,\frac{x_2+x_4}{2}\right)$ the condition that $d_3'<d_3$ is equivalent to
$$
\frac{x_2+x_4}{2}-u<x_3-\frac{x_2+x_4}{2}
\Longleftrightarrow u>x_2+x_4-x_3=:Q_0.
$$

\item
For $d_2$ to overtake $d_3$, we need
\begin{align*}
\left|x_2 -\frac{x_1+u}{2} \right| >\frac{x_2+x_4}{2}-u
\quad \Longleftrightarrow\quad
\begin{cases}
&u>x_1-x_2+x_4=:Q_1\\
& \text{or}\\
&u>\frac{-x_1+3x_2+x_4}{3}=:Q_2
\end{cases}
\end{align*} 

\item
For $d_4$ to overtake $d_3$, we need
\begin{align*}
 \left|x_4 -\frac{u+x_5}{2} \right| >\frac{x_2+x_4}{2}-u
\quad \Longleftrightarrow\quad
\begin{cases}
&u>x_2-x_4+x_5=:Q_3\\
&\text{or}\\
&u>\frac{x_2+3x_4-x_5}{3}=:Q_4
\end{cases}
\end{align*}
As a result, the condition for $d_3$ to be overtaken by some other node, or $d_3'<d_3$ is
\begin{align*}
u&>\min_{j=0,1,2,3,4} Q_j.
\end{align*}
\end{itemize}
Consequently, we must set
\begin{align*}
a&=\max\left\{0, \min \{Q_0,Q_1,Q_2,Q_3,Q_4\}\right\}
\\
&=\max\left\{0, \min\left\{
x_2+x_4-x_3,  x_1-x_2+x_4,  \frac{-x_1+3x_2+x_4}{3},
x_2-x_4+x_5,  \frac{x_2+3x_4-x_5}{3}
\right\}\right\},
\\ b&=x_3.
\end{align*}
Note that we are guaranteed that $ a\le b$. This is trivial when $a=0$; on the other hand, when $a>0$ we have
$$
a\le x_2+x_4-x_3=
\frac{x_2+x_4}{2}-\left[x_3-\frac{x_2+x_4}{2}\right]
<\frac{x_2+x_4}{2}<x_3=b
$$
since $x_3>\frac{x_2+x_4}{2}$.

By substituting $b=x_3$ into the expression for the drift~\eqref{eqDrift}, we get
$$
\Delta=(x_3-a)(x_1+2x_2-4x_3+2x_4+x_5-2a)
$$
and to establish $\Delta\le 0$ it suffices to show
\begin{align}\label{eqner}
x_1+2x_2-4x_3+2x_4+x_5\le 2 a
=2\max\{0,\min\{Q_0,Q_1,Q_2,Q_3,Q_4\}\}
\end{align}
under the assumption that
$$
x_3-\frac{x_2+x_4}{2}>\max\left\{
\left|x_2-\frac{x_1+x_3}{2}\right|,
\left|x_4-\frac{x_3+x_5}{2}\right|
\right\}
$$
that is, equivalently, 
\begin{align}\label{eqx3}
x_3>\max\{Q_1,Q_2,Q_3,Q_4\}.
\end{align}
In order to show~\eqref{eqner} we consider a number of cases.
First, assume that
$x_2+x_4<x_3$. Then $Q_0<0$ and $a=0$. From~\eqref{eqx3} we get that $2 x_3>Q_1+Q_3=x_1+x_5$, thus
$$
x_1+2x_2-4x_3+2x_4+x_5=(x_1+x_5-2x_3)+ 2(x_2+x_4-x_3)<0=a
$$
and~\eqref{eqner} is fulfilled.

The next case is when $\frac{x_2+x_4}{2}<x_3<x_2+x_4$.
We need to verify if all of the following holds:
\begin{align*}
x_1+2x_2-4x_3+2x_4+x_5-2 Q_j\le 0 \quad \text{ subject to }\\
Q_0\ge 0,\ x_3\ge Q_1\ge 0,\ x_3\ge Q_2\ge 0,\ x_3\ge Q_3\ge 0,\ x_3\ge Q_4\ge 0
\end{align*}
and 
\begin{align*}
x_1+2x_2-4x_3+2x_4+x_5\le 0\quad \text{ subject to } \\
\quad Q_j\le 0,\ x_3\ge Q_1,\ x_3\ge Q_2,\ x_3\ge Q_3,\ x_3\ge Q_4 
\end{align*}
for $j=0,1,2,3,4$.
This can be done using Linear Programming method. Thus $\Delta\le 0$.
\end{proof}
${}$

The next statement shows that the metrics provided by $h(x)$, $d(x)$, and $\max_{i\in S} |x_i-x_{i-1}|$, where $x\in \R^N $ are, in fact, equivalent.

\begin{lemma}\label{lembounds}
Let $x=(x_1,\dots,x_N)$ and $\Delta_i(x):=x_i-x_{i-1}$, $i\in S$. Then
\begin{align*}
\begin{array}{rcccl}
d(x) &\le& \ds\max_{i\in S} \left|\Delta_i\right| &\le&  N d(x),\\
2\, d(x)^2 &\le&  h(x)&\le& 6\,N^3\, d(x)^2.
\end{array}
\end{align*}
\end{lemma}
\begin{proof}
Note that $\Delta_1+\dots+\Delta_N=0$ and 
\begin{align*}
h(x)&=\sum_{i\in S} \left[2\Delta_i ^2+(\Delta_i+\Delta_{i+1})^2\right], \\
d(x)&=\frac 12\, \max_{i\in S} \left|\Delta_{i+1}-\Delta_i\right|.
\end{align*}
Let $j$ be such that $d_j(x)=d(x)$, then by the triangle inequality
$$
|\Delta_{j+1}|+|\Delta_j|\ge |\Delta_{j+1}-\Delta_j|=2d(x)
$$
so at least one of the two terms on the LHS $\ge d(x)$, hence  $\max_{i\in S} |\Delta_i|\ge d(x)$.

Now we will show that $\max_{i\in S} |\Delta_i|\le N d(x)$. Indeed, suppose that this is not the case, and w.l.o.g.\ $\Delta_1>N d(x)$. For all $i$ we have 
$\left|\Delta_{i+1}-\Delta_i\right|\le 2d(x)$, hence by induction and the triangle inequality we get
\begin{align*}
&\Delta_2>\left(N-2\right) d(x),\\ 
&\Delta_3>\left(N-4 \right)d(x),\\ &\dots, \
\\
&
\Delta_{N-1}>\left(N-2(N-2) \right) d(x),\\
&\Delta_N>\left(N-2(N-1) \right) d(x).
\end{align*}
As a result, $\Delta_1+\Delta_2+\dots+\Delta_N>\left[N^2-2(1+2+\dots+(N-1))\right]d(x)=N d(x) \ge 0$, which yields a contradiction, since the LHS is identically equal to $0$.

Thus $|\Delta_i|\le N d(x)$, and so 
$|\Delta_i+\Delta_{i+1}|\le 2N d(x)$ for all $i\in S$.
Consequently, $h(x)\le 2N (Nd(x))^2  +N (2Nd(x))^2=6N^3 d(x)^2$. On the other hand, 
$h(x)\ge \ds\max_{i\in S} 2\Delta_i^2\ge 2 d(x)^2$.
\end{proof}
${}$

The following four statements (Lemmas~\ref{lemspeed} and~\ref{lemupdown} and Corollaries~\ref{corr3d} and~\ref{corr_xi_decr}) show that $\xi(t)$ can actually decrease by a non-trivial factor with a positive (and bounded from below) probability.

\begin{lemma}\label{lemspeed}
Suppose that $X(t)=x=(x_1,x_2,x_3,x_4,x_5,\dots)$, and $d_3(x)\ge \max\left\{d_2(x),d_4(x)\right\}$. Let $\mu=\frac{x_2+x_4}{2}$ and $\delta=|x_3-\mu|=d_3(x)$. If $x_3$ is replaced by some $u\in [\mu-\delta/6,\mu+\delta/6]$ then $\Delta_h:=h(X(t+1))-h(X(t))\le - \frac56 \delta^2$. (Note that the Lebesgue measure of $[\mu-\delta/6,\mu+\delta/6] \bigcap [0,1]$ is always at least $\delta/6$; also after this replacement $d_3$ must decrease.)
\end{lemma}
\begin{proof}
Note that the change in $h$ equals
\begin{align*}
\Delta_h=-2(x_3-u) (3u+A),\qquad \text{ where }A= 3x_3-x_1-2x_2-2x_4-x_5.    
\end{align*}
W.l.o.g.\ assume $x_3>\mu$. Then 
$$
x_3-u\ge \mu+\delta -\left(\mu+\frac\delta6\right)=\frac{5}{6}\,\delta.
$$
At the same time, recalling that $d_3(x)\ge \max\{d_2(x),d_4(x)\}$, we obtain that
\begin{align*}
\min_{x_1,\dots,x_5\ge 0} A\qquad \text{ subject to }\qquad
x_3-\mu>\max\left\{\left|x_2-\frac{x_1+x_3}{2}\right|,\left|x_4-\frac{x_3+x_5}{2}\right|\right\}    
\end{align*}
equals $-3\mu+\delta$.
Hence
$$
3u+A\ge 3\left(\mu-\frac\delta6\right)-3\mu+\delta=\frac\delta2
$$
and thus $\Delta_h\le -2\, \frac{5\delta}6 \cdot \frac{\delta}{2}$.
\end{proof}

\begin{lemma}\label{lemupdown}
Suppose that $X(t)=x=(x_1,x_2,x_3,x_4,x_5,\dots)$, and $d_3(x)=d(x)$. Let $\mu=\frac{x_2+x_4}{2}$ and $\delta=|x_3-\mu|=d_3(x)$. Given that $x_3>\mu$, if $x_3$ is replaced by some $u\notin [\mu-3\delta,x_3]$ then $d_3(x')>d_3(x)$ and $d_3(x')$ is still the largest of $d_i(x')$, where $x'=(x_1,x_2,u,x_4,x_5,\dots)$. The same conclusion holds if $x_3<\mu$ and $x_3$ is replaced by some $u\notin [x_3,\mu+3\delta]$.
\end{lemma}
Before presenting the proof of Lemma~\ref{lemupdown}, we state the obvious
\begin{corollary}\label{corr3d}
Let $\delta=d(\tilde X(s))$. If $i=\jmath(\tilde X(s))$ then
$$
\tilde X_i(s+1)\in [\tilde X_i(s)-4\d,\tilde X_i(s)+4\d]
$$ 
(and if $i\ne \jmath(\tilde X(s))$ then trivially $X_i(s+1)=X_i(s)$). Hence we always have
$$
\max_{i\in S} \left|\tilde X_i(s+1)-\tilde X_i(s)\right|\le 4 \d.
$$
\end{corollary}
(Note that in Corollary~\ref{corr3d} we have $4\d$ for the following reason: the newly accepted point can deviate from $\mu$ by at most $3\d$ by Lemma~\ref{lemupdown}, while $|\tilde X_i(s)-\mu|=\d$.)

The next implication of Lemma~\ref{lemupdown} requires a bit of work.
\begin{corollary}\label{corr_xi_decr}
Let $\rho=1-\frac{5}{36\, N^3}<1$. Then
$$
\P\left(\xi(s+1)\le \rho \xi(s)\| \tF_s \right)\ge \frac 1{48}.
$$
\end{corollary}
\begin{proof}[Proof of Corollary~\ref{corr_xi_decr}]
From Corollary~\ref{corr3d} we know that given $x=\tilde X(s)$, the allowed range for the newly sampled point to be in $\tilde X(s+1)$ is at most $8 \d$ where $\d=d(x)$. At the same time if the newly sampled point falls into the interval $[\mu-\d/6,\mu+\d/6]$ (see Lemma~\ref{lemupdown}), at least half of which lies in $[0,1]$, then $\xi(s+1)-\xi(s)\le -\frac 56 \d^2$; the probability of this event is no less than $\frac{\d/6}{8\d}=\frac 1{48}$. Since $\xi(s)=h(x)$ and by Lemma~\ref{lembounds} we have $d(x)^2\ge \frac{h(x)}{6N^3}$, the inequality
$\xi(s+1)-\xi(s)\le -\frac 56 \d^2$ implies
$\xi(s+1)-\xi(s)\le -\frac 5{36N^3} \xi(s)$.
\end{proof}

\begin{proof}[Proof of Lemma~\ref{lemupdown}]
By symmetry, it suffices to show just the first part of the statement. First, observe that
\begin{align}\label{eqdd}
d_j(x')&=d_j(x)\le d_3(x)\text{ for } j\in S\setminus \{2,3,4\}; 
\nonumber
\\
d_2(x')&= \left|\left(\frac{x_1+x_3}{2}-x_2\right) +\frac{u-x_3}2\right|
\le d_2(x)+\left| \frac{u-x_3}2\right|\le 
d_3(x)+\left| \frac{u-x_3}2\right|.
\end{align}

If $u>x_3>\mu$, then from~\eqref{eqdd}
\begin{align*}
d_3(x')&=u-\frac{x_2+x_4}{2}>x_3-\frac{x_2+x_4}{2}=d_3(x);\\
d_2(x')&\le d_3(x)+\left| \frac{u-x_3}2\right|
=
d_3(x')-(u-x_3)+\left| \frac{u-x_3}2\right|
=d_3(x')-\left| \frac{u-x_3}2\right|<d_3(x');\\
d_4(x')&<d_3(x')\quad \text{ (by the same argument as $d_2$)}
\end{align*}
so indeed $d_3(x)<d_3(x')=\max_{i\in S} d_i(x')$.

On the other hand, if $u<\mu-3\delta<x_3=\mu+\delta$, then $d_j$ for $j\in S\setminus\{2,3,4\}$ still remain unchanged, but
\begin{align*}
d_3(x')&=\mu-u>3\delta>d_3(x);\\
d_2(x')&\le d_3(x)+\left| \frac{u-x_3}2\right|
=\delta+ \frac{x_3-u}2=\delta+ \frac{x_3-\mu}2+\frac{\mu-u}2
= \frac{3\delta}2+\frac{\mu-u}2
\\ &
< \frac{\mu-u}2+\frac{\mu-u}2
=d_3(x')
;\\
d_4(x')&< d_3(x')\quad \text{ (by the same argument as $d_2$)}
\end{align*}
hence $d_3(x)<d_3(x')=\max_{i\in S} d_i(x')$ in this case as well.
\end{proof}
${}$

At the same time, it turns out that $\xi(t)$  cannot increase too much in one step, as follows from

\begin{lemma}\label{lem_up_jump}
There is a non-random $r>0$ such that for all $s$ we have
$\xi(s+1)\le r \xi(s)$.
\end{lemma}
\begin{proof} By Corollary~\ref{corr3d} it follows that the worst outlier (w.l.o.g.\ $x_3$) can be replaced only by a point at most at the distance $4\delta$ from $x_3$ at time $\nu_{s+1}$. Let the new value of the fitness at node~$3$ be $x_3+v$, $|v|\le 4\delta$. The change in the Lyapunov function is given by
\begin{align}\label{eq_xi_s+1}
\xi(s+1)-\xi(s)&=\left[2((x_3+v)-x_2)^2+2((x_3+v)-x_4)^2+((x_3+v)-x_1)^2+((x_3+v)-x_5)^2\right]
 \nonumber
\\ &-\left[2(x_3-x_2)^2+2(x_3-x_4)^2+(x_3-x_1)^2+(x_3-x_5)^2\right]
 \nonumber
\\
& =(12x_3-2x_2-2x_4-4x_1-4x_5)\, v+6\,v^2
\end{align}
Since
\begin{align*}
 \left|12x_3-2x_2-2x_4-4x_1-4x_5\right|&=\left|8\left(x_2-\frac{x_1+x_3}2\right)+8\left(x_4-\frac{x_5+x_3}2\right)
+20\left(x_3-\frac{x_2+x_4}2\right)\right|
\\ &\le 8\d+8\d+20\d=36\d
\end{align*}
from~\eqref{eq_xi_s+1} and the fact that $\delta=d(\tilde X(s))\le\sqrt{\frac{\xi(s)}2}$ by Lemma~\ref{lembounds}
$$
|\xi(s+1)-\xi(s)|\le 36\d\times 4\d+ 6\,(4\delta)^2= 240\d^2\le 120 \xi(s),
$$
so we can take $r=121$.
\end{proof}
${}$

Finally, we want to show that, roughly speaking, one does not have to wait for too long before~$\xi(t)$ increases or decreases by a {\em substantial} amount.

\begin{lemma}\label{lem_finite_mean}
Fix some $k>1$ and  $s_0>0$.  Let $\tau_1=\inf\{s>0:\ \xi(s_0+s)\le \xi(s_0)/k\}$ and $\tau_2=\inf\{s>0:\ \xi(s_0+s)\ge k\xi(s_0)\}$. Then~$\tau=\min(\tau_1,\tau_2)$, given $\tF_{s_0}$, is stochastically smaller than some random variable with a finite mean, the distribution of which does not depend on anything except~$N$ and~$k$.
\end{lemma}
\begin{proof}
Fix a positive integer $L$. For each $t\ge s_0$ 
define 
$$
B_t=\left\{\xi(t+L)\le \frac{\xi(t)}{k^2}\right\}.
$$
It suffices to show that $\P(B_t| \tF_t)\ge p$ for some $p>0$ uniformly in $t$, since for $j=0,1,2,\dots$
\begin{align*}
B_{s_0+jL} & \subseteq \{\xi(s_0+jL)<k\xi(s_0)\text{ and } \xi(s_0+(j+1)L)<\xi(s_0)/k \}
\cup \{\xi(s_0+jL)\ge k \xi(s_0)\}
\\
&\subseteq \{\tau_1\le (j+1)L\} \cup \{\tau_2\le jL\}
\subseteq \{\tau\le (j+1)L\}.
\end{align*}
which, in turn, would imply that $\tau$ is stochastically smaller than $L$ multiplied by a geometric random variable with parameter $p=p(N,k)$.

To show that $\P(B_t\|\ \tF_t)\ge p$,  note that by Corollary~\ref{corr_xi_decr},
$$
\P(B_{m}^* \| \tF_{m-1})\ge \frac 1{48},
\quad \text{where } B_{m}^*=\left\{\xi(m)<\rho \xi(m-1)\right\},
\quad \rho=1-\frac{5}{36 N^3}.
$$
Let $L$ be so large that $\rho^L<1/k^2$. Then, on one hand,
$$
\bigcap_{m=1}^{L}B_{t+m}^*\subseteq B_t
\text{ whence } 
\P\left(B_t \| \tF_t\right)\ge 
\P\left(\bigcap_{m=1}^{L}B_{t+m}^*   \| \tF_t \right),
$$
while on the other hand
$$
\P\left(\bigcap_{m=1}^{L}B_{t+m}^* \| \tF_t\right)\ge \frac 1{48^L}=:p
$$
which depends on $N$ and $k$ only.
\end{proof}
${}$

The proof of the next statement, which completes the first part of the proof of the main theorem, requires a bit more work than that of Lemma~2.4 in~\cite{GVW}. In fact, we will prove a stronger statement (Corollary~\ref{corr_expfast}) later, however, it is still useful to see a fairly quick proof of the following

\begin{lemma}\label{lem_xi->0}
$\xi(s)\to 0$ a.s.\ as $s\to\infty$ (and as a result $\Delta_i(\tilde X(s))\to 0$ a.s. and $d(\tilde X(s))\to 0$ a.s. as $s\to\infty$).
\end{lemma}
\begin{proof}
From Lemma~\ref{lemsuperm} it follows that $\xi(s)$ converges a.s.\ to a non-negative limit, say $\xi_\infty$. Let us show that $\xi_\infty=0$.
From Corollary~\ref{corr_xi_decr} we have
\begin{align}\label{eqdown}
\P\left(\xi(s+1)\le \rho \xi(s)\, |\, {\cal F}_s\right)\ge \frac 1{48}.
\end{align}

Fix an $\eps>0$ and a $T\in\N$. Let $\sigma_{\eps,T}=\inf\{s\ge T:\  \xi(s)\le \eps\}$. Then~\eqref{eqdown} implies
\begin{align*}
\P(A_{s+1}\,|\, {\cal F}_s)&\ge \frac{1_{s<\sigma_{\eps,T}}}{48},\quad
\text{where }A_{s+1}=\left\{\xi(s+1)\le \xi(s)-(1-\rho)\eps \right\}
\end{align*}
(Compare this with the inequality~(2.18) in~\cite{GVW}).
From the non-negativity of $\xi(s)$, 
we know that only finitely many of $A_s$ can occur. By the Levy's extension to the Borel-Cantelli lemma, we get that $\sum_{s=T}^\infty \P(A_{s+1}\,|\, {\cal F}_s)<\infty$ a.s., and hence $\sum_{s=T}^\infty 1_{s<\sigma_{\eps,T}}<\infty$. This, in turn, implies that $\sigma_{\eps,T}<\infty$ a.s. Consequently, since $T$ is arbitrary,
$$
\liminf_{s\to\infty}\xi(s)\le \eps\quad \text{a.s.}
$$
Since $\eps>0$ is also arbitrary and $\xi(s)$ converges, $\lim_{s\to\infty} \xi(s)=\liminf_{s\to\infty}\xi(s)=0$ a.s.
\end{proof}
${}$

The next general statement may be known, but since we could not find it in the literature, we present its fairly short proof. We need it in order to show that $\xi(t)$ converges to zero quickly.  

\begin{prop}\label{prop_mart}
Suppose that $\xi(s)$ is a positive bounded supermartingale with respect to a filtration~$\tF_s$. Suppose there is a constant $r>1$ such that $\xi(s+1)\le r \xi(s)$ a.s.\ and that for all $k$ large enough 
the stopping times
$$
\tau_s=\inf\{t>s:\ \xi(t)>k \xi(s)\text{ or } \xi(t)<k^{-1} \,\xi(s)\}
$$
are stochastically bounded above by some finite--mean random variable $\bar\tau>0$, which depends on~$k$ only (and, in particular, independent of $\tF_s$). Let $\mu=\E \bar \tau<\infty$. Then 
\begin{align*}
\limsup_{s\to\infty} \frac{\ln \xi(s)}{s}\le -\frac 1{4\mu}<0\qquad\text{a.s.}
\end{align*}
\end{prop}

\begin{proof}
First, observe that by the Optional Stopping Theorem
\begin{align}\label{eqxiless}
\E (\xi(\tau_s)\|\tF_s) \le \xi(s)
\end{align}
(where $\tau_s<\infty$ a.s.\ by the stochastic dominance condition) while, on the other hand,
\begin{align}\label{eqximore}
\E (\xi(\tau_s)\|\tF_s)&=\E (\xi(\tau_s),\xi(\tau_s)>k \xi(s)\|\tF_s)+\E (\xi(\tau_s),\xi(\tau_s)<k^{-1}\ \xi(s)\|\tF_s)\nonumber
\\ &
\ge \E (\xi(\tau_s),\xi(\tau_s)>k \xi(s)\|\tF_s)
\ge k\xi(s) \cdot \P(\xi(\tau_s)> k \xi(s)\|\tF_s).
\end{align}
From~\eqref{eqxiless}  and \eqref{eqximore} we conclude
\begin{align}\label{eq_def_p}
p:=\P(\xi(\tau_s)>k \xi(s)\|\tF_s)<\frac 1{k}.
\end{align}
Now let us define a sequence of stopping times as follows: $\eta_0=0$ and for $n=1,2,\dots$,
\begin{align*}
\eta_{n}=\inf\left\{s>\eta_{n-1}:\ \xi(s)>k \xi(\eta_{n-1}) \text{ or } \xi(s)<k^{-1}\ \xi(\eta_{n-1}) \right\}
\end{align*}
and let
$$
N_s=\max\{n:\ \eta_n\le s\}.
$$
From the definition of the stopping times $\eta$, it follows 
\begin{align}\label{eq_xi_k_N}
\xi(s)\le k \xi(\eta_{N_s}),\qquad  \xi(\eta_{n+1})\le rk \xi(\eta_n).
\end{align}

Consider now the sequence of random variables $\xi(\eta_n)$. From~\eqref{eq_def_p} and~\eqref{eq_xi_k_N} we obtain that 
$\log_k \frac{\xi(\eta_n)} { \xi(\eta_{n-1})}$ is stochastically bounded above by a random variable $X_n\in\{-1,1+\log_k r\}$ such that
\begin{align*}
1-\P(X_n=-1)=\P(X_n=1+\log_k r)=\frac 1{k}
\end{align*}
yielding
$$
\E X_n =\frac {2+\frac{\ln r}{\ln k}}{k} -1 =:g(r,k);
$$
we can also assume that $X_n$ are i.i.d. One can choose $k>1$ so large\footnote{if $r>4.1$, then $k=\ln(r)$ will be sufficient.}   that  $g(r,k)<-\frac 12$. Then, by the Strong Law applied to $\sum_{i=1}^n X_i$, we get
$$
\limsup_{n\to\infty} \frac{\log_k \xi(\eta_n)}{n} \le \limsup_{n\to\infty} \frac{X_1+\dots+X_n}{n} <  -\frac12
\qquad\text{a.s.}
$$

From the condition of the proposition we know that the differences $\eta_n-\eta_{n-1}$, $n=1,2,\dots,$ are stochastically bounded by independent random variables with the distribution of $\bar\tau$ with $\E \bar\tau =:\mu<\infty$. Then by the Strong Law for renewal processes (see e.g.~\cite{DUR}, Theorem~I.7.3) applied to the sum of independent copies of $\bar\tau$, we get
\begin{align}\label{eq_renewal}
\liminf_{s\to\infty}\frac{N_s}{s}\ge \frac 1{\mu}\qquad\text{a.s.}
\qquad
\Longrightarrow
\qquad
s\le 2\mu N_s \text{ for all large enough }s.
\end{align}

Combining~\eqref{eq_xi_k_N} and \eqref{eq_renewal}, we get
\begin{align*}
\limsup_{s\to\infty} \frac{\log_k \xi(s)}{s}
&\le
\limsup_{s\to\infty} \frac{\log_k \left( k\xi(\eta_{N_s}) \right)  }{s}
=
\limsup_{s\to\infty} \frac{\log_k \xi(\eta_{N_s})}{s}
\\ &
\le \limsup_{s\to\infty} \frac{\log_k \xi(\eta_{N_s})}{2\mu N_s}
= \frac{1}{2\mu} \limsup_{n\to\infty} \frac{\log_k \xi(\eta_n)}{n}
\le -\frac{1}{4\mu} \qquad\text{a.s.}
\end{align*}
since 
$N_s\to\infty$ when $s\to\infty$ a.s.
\end{proof}

The next statement strengthens Lemma~\ref{lem_xi->0}.
\begin{corollary}\label{corr_expfast}
$\xi(s)\to 0$ exponentially fast as $s\to\infty$.
\end{corollary}
\begin{proof}
The statement follows immediately from Proposition~\ref{prop_mart}: the bound for $r$ we have by Lemma~\ref{lem_up_jump}; the other condition follows from Lemma~\ref{lem_finite_mean}.
\end{proof}
${}$

Now we are ready to finish the proof of the main statement.

\begin{proof}[Proof of Theorem~\ref{thm_main_alt}]
According to Corollary \ref{corr_expfast} there exist $a,b>0$ which are a.s.\ finite and such that $\xi(t)\le ae^{-bt}$. If we take $s_0$ such that $ae^{-bs}\le \epsilon$ for all $s\ge s_0$ then if $s_0\le s <t$, 
\begin{align}\label{eqcauchy}
|\tilde X_i(t)-\tilde X_i(s)|&\le  \sum_{k=s+1}^{t} 4\,d(\tilde X(k))\le  \sum_{k=s+1}^{t} \sqrt{8\xi(k)} \nonumber \\ 
&\le \sqrt{8\epsilon} \sum_{k=s+1}^{t} e^{-bk/2}\le  \frac{\sqrt{8\epsilon}}{1-e^{-b/2}},
\end{align}
where we used Corollary~\ref{corr3d} in the first inequality and Lemma~\ref{lembounds} in the second inequality. We can thus conclude that $\{\bar X_i(t)\}_t$ is a Cauchy sequence in the a.s.\ sense; therefore the limit $\bar X_i(\infty)=\lim_{t\to\infty} \tilde X_i(t)$ exists a.s. Moreover, by letting $t\to\infty$ in~\eqref{eqcauchy}, we get that
$|\tilde X_i(s)-\tilde X_i(\infty)|\le C e^{-bs/2}$ for some $C>0$.

Furthermore, assuming w.l.o.g.\  that $i<j$,
\begin{align*}
|\bar X_i(\infty)-\bar X_j(\infty) | &=\lim_{t\to\infty} |\tilde X_i(t)-\tilde X_j(t)| 
\le \lim_{t\to\infty} \sum_{k=i+1}^{j} \left|\Delta_k( \tilde X(t) )\right| =0
\end{align*}
by Lemma~\ref{lem_xi->0}, which completes the proof.
\end{proof}

\section{Discussion and open problems}
One may be interested in the speed of convergence, established in   Theorem~\ref{thm_main_alt}. In Lemma~\ref{lem_up_jump} we can take $r=121$ and from the proof of Proposition~\ref{prop_mart}, $k=\ln r=\ln (121)=2\ln(11)$ will be sufficient. Then, for Lemma~\ref{lem_finite_mean}, find $L$ such that
$$
\left(1-\frac{5}{36N^3}\right)^L<\frac{1}{23}<\frac1{k^2}
$$
We can take, e.g.,
$$
L\approx 7.2 N^3 \cdot \ln(23)\approx 22.6 N^3
$$
This, in turn, will provide a bound on $\mu=\E\bar \tau\le \frac{ L}p=L\cdot 48^L$ for Proposition~\ref{prop_mart}, and hence the speed of the convergence for large $s$:
$$
2\, [d(\tilde X(s))]^2\le h(\tilde X(s))=\xi(s)\le k^{-\frac{s}{4\mu}}\le \exp\left\{\ds-\frac{s}{8\, L\, 48^L\, \ln(11)}\right\}
\approx
\exp\left\{\ds-\frac{s}{433 \cdot 10^{38 N^3}}\right\}
$$
This bound is, however,  far from the optimal one. The simulations seem to indicate that, depending on $N$, $$\xi(s)\sim e^{-\rho_N s},$$ where e.g.\
 $\rho_5\in (0.47,0.77)$,
 $\rho_{10}\in (0.14,0.23)$,
 $\rho_{20}\in (0.02,0.03)$,
 $\rho_{40}\in (0.003,0.006)$, suggesting that (a) $\rho_N$ can be, in fact, random, and (b) the average value of $\rho_N$ decays roughly like $5/N^{2}$. We leave the study of the properties of $\rho_N$ for further research.

We believe that the convergence, described by Theorems~\ref{thm_main} and~\ref{thm_main_alt} holds for a much more general class of replacement distributions $\zeta$, not just uniform; for example, for the continuous distributions with the property that their density is uniformly bounded away from zero. Unfortunately, our proof is based on the construction of the Lyapunov function which cannot be easily transferred to  other cases (obviously, it will work for any $\zeta\sim U[a,b]$, where $a<b$). 

One can also attempt to generalize the theorems for more general graphs as described in Remark~\ref{rem1}; this should be done, however, with care, as it will not work for all the distributions (see Remark~\ref{rem2}).

%

\end{document}